\documentclass[12pt,reqno]{amsart}
\usepackage{amssymb, a4wide, amsmath, mathtools,xy}
\xyoption{all}
\usepackage{latexsym,pdfsync,xcolor,graphicx}
\usepackage{multirow}

\usepackage[T1]{fontenc}

\usepackage[utf8]{inputenc}

\usepackage{array}
\usepackage{upgreek}
\usepackage{pstricks-add}
\usepackage{enumerate}
\usepackage{orcidlink}

\usepackage[english]{babel}	
\usepackage{comment}
\allowdisplaybreaks

\usepackage{pgf,tikz}
\usepackage{wrapfig}
\usetikzlibrary{arrows}

\usepackage{float}
\usepackage{appendix}

\usepackage{hyperref}
\hypersetup{colorlinks,
    citecolor=black,
    filecolor=black,
    linkcolor=black,
    urlcolor=black}

\theoremstyle{plain}%
\newtheorem{theorem}{Theorem}[section]
\newtheorem{corollary}[theorem]{Corollary}%
\newtheorem{conjecture}[theorem]{Conjecture}%

\theoremstyle{definition}

\theoremstyle{remark}
\newtheorem{remark}[theorem]{Remark}

\newcommand{\E}{\mathcal{E}}

\DeclareMathOperator{\spa}{\textup{\textsf{span}}}

\DeclareMathOperator{\rank}{rank}

\numberwithin{equation}{section}
\title[A note on complete evolution algebras]{A note on complete evolution algebras}

\author[X. García-Martínez]{Xabier García-Martínez\orcidlink{0000-0003-1679-4047}}
\author[A. Pérez-Rodríguez]{Andrés Pérez-Rodríguez\orcidlink{0009-0007-1095-5328}}

\address[Xabier García Martínez, Andrés Pérez Rodríguez]{CITMAga \& Department of Mathematics, Universidade de Santiago de Compostela, 15782 Santiago de Compostela, Spain}
\email{xabier.garcia@usc.es}

\email{andresperez.rodriguez@usc.es}



\begin{document}
	
	\begin{abstract}
     This short note provides positive answers to two conjectures of Camacho, Khudoyberdiyev, and Omirov on the classification of complete evolution algebras. Our approach is based on analysing the solution set of a generic non-linear polynomial system of equations using elementary tools from algebraic geometry. We also obtain new results on subalgebras and idempotents of evolution algebras, and conclude by proposing a conjecture that may characterise solvable evolution algebras.
    \end{abstract}
	
\maketitle

\section{Introduction}

Evolution algebras are commutative but typically non-associative algebras. They were introduced by J.~P.~Tian and P.~Vojtěchovský in 2006 \cite{TV_06} as an algebraic framework for modelling non-Mendelian inheritance. Two years later, Tian’s monograph \cite{Tian_08} further developed these ideas, highlighting the relevance of evolution algebras for describing the reproductive behaviour of asexual organisms through self-replication rules. This biological motivation naturally leads to the following algebraic notion. Formally, an evolution algebra over a field $\mathbb{K}$ is a $\mathbb{K}$-algebra $\mathcal{E}$ which admits a distinguished basis $B=\{e_1,\dots,e_n,\dots\}$, called a \textit{natural basis}, such that $e_ie_j=0$ for all $i\neq j$.

The theory of evolution algebras is currently a very active area of research. In particular, one of the main challenges is their classification. This problem has been addressed in low dimensions (see \cite{CSV_17,CLOR_14}), and later for the nilpotent case (see \cite{EL_16,HA_15}). As expected, complete classifications are hard to obtain, and most contributions instead focus on analysing specific intrinsic structural properties and classifying the algebras that satisfy them.

In this note we focus on complete evolution algebras. To better understand what this property implies, let us briefly recall how subalgebras of evolution algebras behave. The key point is that evolution algebras are not, in general, closed under taking subalgebras (see \cite[Example~1.4.1]{thesis_yolanda}). Consequently, three notions naturally arise when studying subalgebras: ordinary subalgebras, subalgebras admitting a natural basis, and subalgebras admitting a natural basis that extends to a natural basis of the entire algebra. Regarding the latter notion, an evolution algebra $\E$ is said to be \textit{complete} if any subalgebra of $\E$ admits a natural basis which can be extended to one of $\E$. This notion was investigated in \cite{CKO_19}, which concludes with two conjectures concerning both the structure and the classification of complete evolution algebras.

The purpose of this note is to prove these two conjectures by reducing them to a result on the existence and shape of solutions of a certain non-linear polynomial system, established using elementary tools of algebraic geometry, and to explore several further consequences that follow from it.

The text is structured into three sections. Following this introduction, Section~\ref{sec_2} is devoted to the proof of the key result, which guarantees the existence of a solution with at least two non-zero components to a certain polynomial system. This result underlies all the consequences for evolution algebras presented in Section~\ref{sec_3}, where we establish results concerning subalgebras and idempotents, as well as the classification of the so-called complete evolution algebras.

\section{The key result}\label{sec_2}
The main goal of this section is to prove a relaxed version of \cite[Conjecture 5.1]{CKO_19}, which will be sufficient for the statements about evolution algebras that we aim to establish later. In its original form, the conjecture says that given a complex invertible matrix $A=(a_{ij})_{1\leq i,j\leq n}$, the system of equations
\begin{align}\label{sist_1}
	\begin{pmatrix}
		x_1^2 \\
		x_2^2 \\
		\vdots\\
		x_n^2 
	\end{pmatrix}=
	\begin{pmatrix}
		a_{11} & a_{12} & \cdots & a_{1n}\\
		a_{21} & a_{22} & \cdots & a_{2n}\\
		\vdots & \vdots & \ddots & \vdots\\
		a_{n1} & a_{n2} & \cdots & a_{nn}
	\end{pmatrix}
	\begin{pmatrix}
		x_1 \\
		x_2 \\
		\vdots\\
		x_n
	\end{pmatrix},
\end{align}
admits a solution $(x_1,x_2,\dots,x_n)$ such that $x_i\neq0$ for all $i$. Nevertheless, we now show that there exists a solution with at least two non-zero components, without requiring all components to be non-zero.

\begin{theorem}\label{th:sol_sist}
	Let $A=(a_{ij})_{1\leq i,j\leq n}$ be a complex invertible matrix. Then, the system of equations~\eqref{sist_1}	
	always admits a non-trivial solution. In particular, there exists a solution in which at least two coordinates are non-zero. 
\end{theorem}
\begin{proof}
	For convenience, define the polynomials $F_i(x_1,\dots,x_n)=x_i^2-\sum_{j=1}^{n}a_{ij}x_j$ for all $i=1,\dots,n$. Note that the system $\{F_i=0\}_{i=1}^n$ coincides with \eqref{sist_1}. Moreover, we homogenise them to obtain homogeneous polynomials defining hypersurfaces in $\mathbb{P}^n(\mathbb{C})$:
	\[
        \widetilde{F}_i(x_1,\dots,x_n,t)=x_i^2-t\sum_{j=1}^na_{ij}x_j.
    \]
	 
	We first show that the hypersurfaces $\widetilde{F}_1,\dots,\widetilde{F}_n$ share no common component.
	Let us assume that the homogeneous polynomials $\widetilde{F}_1,\dots, \widetilde{F}_n$ admit a common non-constant divisor. Let $g(x,t)=g(x_1,\dots,x_n,t)$ be an irreducible homogeneous polynomial of degree $d\geq1$ dividing all $\widetilde{F}_i$. If we consider the specialisation $t=0$, then $\widetilde{F}_i(x,0)=x_i^2$, so the polynomial $g(x,0)$ must divide $x_i^2$ for every $i$. Since the variables $x_1,\dots,x_n$ are algebraically independent, we have $\gcd(x_1^2,\dots,x_n^2)=1$, and therefore $g(x,0)$ must be a non-zero constant, which yields a contradiction with the fact that $g$ was homogeneous of degree $d\geq1$ because such a polynomial necessarily vanishes at the origin, whereas a non-zero constant does not.

	Since $\widetilde{F}_1,\dots,\widetilde{F}_n$ share no common components, Bézout's theorem can be applied. Hence, these $n$ hypersurfaces intersect in a finite set, in particular, in $2^n$ points of $\mathbb{P}^n(\mathbb{C})$, counting their multiplicities. Moreover, it is easy to check that there are no intersection points at infinity. Indeed, if $t=0$, then it forces $x_{1}=\dots=x_n=0$. This does not represent any projective point, and therefore all intersection points lie in the affine space $\mathbb{A}^n(\mathbb{C})$ (after rescaling $t=1$).
	
	We now claim that points with less than two non-zero coordinates have multiplicity one. A point has multiplicity one precisely when the hypersurfaces intersect transversely at that point. This happens exactly when the Jacobian matrix has full rank at that point. The Jacobian matrix of the affine system is
	\[
	J(x_1,\dots,x_n)=\left(\frac{\partial F_i}{\partial x_i}\right)=
	\begin{pmatrix}
		2x_1-a_{11} & -a_{12} & -a_{13} & \cdots & -a_{1n} \\
		-a_{21} & 2x_2-a_{22} & -a_{23} & \cdots & -a_{2n}  \\
		-a_{31} & -a_{32} & 2x_3-a_{33} & \cdots & -a_{3n}  \\
		\vdots & \vdots & \vdots & \ddots & \vdots  \\
		-a_{n1} & -a_{n2} & -a_{n3} & \cdots & 2x_n-a_{nn} 
	\end{pmatrix}.
	\]
	At the origin, $J(0,\dots,0)=-A$. Since $A$ is invertible by hypothesis, the origin has multiplicity one. Let us consider a possible solution of the form $(x_1,0,\dots,0)$ with $x_1\neq0$. Substituting into the system forces that necessarily $x_1=a_{11}$ and $a_{21}=\dots=a_{n1}=0$. Plugging this into the Jacobian shows that
	\[
	\det J(x_1,0,\dots,0)=a_{11}\begin{vmatrix}
		-a_{22} & \cdots & -a_{2n}\\
		\vdots & \ddots & \vdots\\
		-a_{n2} & \cdots & -a_{nn}
	\end{vmatrix}=-\det(A)\neq0.
	\]
	Hence, this point also has multiplicity one. The same reasoning applies for any point with only one non-zero coordinate.
	
	Finally, note that there are at most $n+1$ points with less than two non-zero entries, and all of them have multiplicity one. Since $n+1<2^n$ for all $n\geq2$, not all the intersection points can be of that form. Therefore, at least one solution must have at least two non-zero components.
\end{proof}

\begin{remark}
	The previous result does not necessarily hold over fields that are not algebraically closed, such as $\mathbb{R}$. For instance, consider the matrix
	\[
	A=\begin{pmatrix}
		1 & -2 & -3 \\ 0 & 0 & 1 \\ 0 & 1 & 1
	\end{pmatrix}.
	\]
	A computer-assisted calculation shows that the only real solution is $(1,0,0)$, which does not have at least two non-zero components.
\end{remark}

\section{Consequences on evolution algebra structures}\label{sec_3}

The previous theorem yields several noteworthy implications for the structure theory of evolution algebras. Before presenting these consequences, we fix some notation. Throughout this section, we work with complex finite-dimensional evolution algebras. Given a natural basis $B=\{e_1,\dots,e_n\}$ of an evolution algebra $\E$, the scalars $a_{ij}\in\mathbb{C}$ satisfying $e_i^2=\sum_{j=1}^na_{ij}e_j$ are called the \textit{structure constants} of $\mathcal{E}$ relative to $B$. The matrix~$M_B(\mathcal{E})=(a_{ij})_{i,j=1}^n$ is said to be the \textit{structure matrix} of $\mathcal{E}$ relative to $B$. Finally, recall that an evolution algebra $\E$ is called \textit{regular} (or \textit{perfect}) if $\E=\E^2$, or equivalently, any of its structure matrices is invertible.

\subsection{Existence of non-zero proper subalgebras}
The system~\eqref{sist_1} does not only arise in \cite{CKO_19}, it appears naturally in \cite[Lemma~1]{LP_25_regular}, where a one-to-one correspondence is established between all non-zero one-dimensional subalgebras of a regular evolution algebra $\E$ and the non-trivial solutions of the system 
\begin{equation}\label{sist_2}
	\begin{pmatrix}
		x_1^2 \\ \vdots \\ x_n^2
	\end{pmatrix}=\big(M_B(\mathcal{E})^t\big)^{-1}\begin{pmatrix}
		x_1 \\ \vdots \\ x_n
	\end{pmatrix}.
\end{equation}
 Since $\E$ is regular, the structure matrix $M_B(\mathcal{E})$ is invertible, and therefore Theorem~\ref{th:sol_sist} applies directly. We thus have the following straightforward consequence.
\begin{theorem}
	Every complex evolution algebra $\E$ admits a non-zero proper subalgebra.
\end{theorem}
\begin{proof}
	If an evolution algebra $\E$ is regular, then the result follows straightforwardly from Theorem~\ref{th:sol_sist}. If $\E$ is abelian, i.e., $\E^2=0$, then the result follows automatically from the fact that every subspace of an abelian algebra is a subalgebra. Otherwise, if $\E$ is neither regular nor abelian, then $0\subsetneq\E^2\subsetneq\E$, so $\E^2$ is a proper non-zero subalgebra.
\end{proof}
It is worth mentioning that simple evolution algebras---those with no non-zero proper ideals---have already been characterised: they are precisely the regular algebras whose associated digraph is strongly connected. Thus, the relevance of the previous result lies in the fact that no analogous characterisation for subalgebras is possible, since every evolution algebra admits at least one non-zero proper subalgebra. 

\subsection{Characterising complete evolution algebras} We show that the two conjectures concerning complete evolution algebras formulated in \cite{CKO_19}, and stated below for completeness, are correct.

\begin{theorem}[{\cite[Conjecture~5.2]{CKO_19}}]\label{conj_1}
	Let $\E$ be a regular evolution algebra of dimension greater than one. Then, $\E$ is not complete.
\end{theorem}
\begin{theorem}[{\cite[Conjecture~5.3]{CKO_19}}]\label{conj_2}
	Let $\E$ be an $n$-dimensional non-nilpotent complete evolution algebra. Then, $\E$ is isomorphic to one of the following pairwise non-isomorphic algebras:
	\[\{e_1^2=e_1\}\oplus \mathbb{C}^{n-1}\quad\text{or}\quad\{e_1^2=e_1\}\oplus \widetilde{\E}\oplus\mathbb{C}^{n-s-1},\]
	where $\widetilde{\E}$ is an $s$-dimensional evolution algebra with maximal index of nilpotency and $\mathbb{C}^k$ denotes the $k$-dimensional zero evolution algebra.
\end{theorem}

\begin{proof}[Outline of both proofs]
The proofs of both results in \cite{CKO_19} rely on the assumed validity of~\cite[Conjecture~5.1]{CKO_19}. The arguments of both proofs require the existence of a solution $(x_1,\dots,x_n)$ of the system~\eqref{sist_1} with all coordinates non-zero, which then yields a one-dimensional subalgebra $\spa\{x_1e_1+\dots+x_ne_n\}$ that cannot be extended to a natural basis of the whole algebra. 

However, we do not actually need all coordinates to be non-zero: having just two non-zero coordinates already suffices. To explain this, we recall the notion of \textit{natural vector} introduced in \cite{BCS_22}, meaning a vector that can be extended to a natural basis of the whole algebra. As shown in \cite[Theorem~2.4]{BCS_22}, if $u=x_1e_1+\dots+x_ne_n$ is an idempotent, $u^2=u\neq0$, then $u$ is a natural vector if and only if $\rank\big(\{e_i^2\colon x_i\neq0\}\big)=1$. In our situation, this condition holds only when exactly one $x_i$ is non-zero. Note that the system~\eqref{sist_1} involves an invertible matrix, which comes from the structure matrix of a regular evolution algebra, and therefore $\{e_1^2,\dots,e_n^2\}$ is linearly independent. Hence, any idempotent supported on at least two basis elements whose squares are linearly independent fails to be a natural vector, giving exactly the obstruction required in the proofs.  
\end{proof}

Note that Theorems~\ref{conj_1} and \ref{conj_2} together with \cite[Theorem~4.2]{CKO_19} entirely classify complete evolution algebras.
\begin{corollary}
    Let $\E$ be an $n$-dimensional complete evolution algebra. Then, $\E$ is isomorphic to one of the following pairwise non-isomorphic algebras:
	\[\{e_1^2=e_1\},\qquad\widetilde{\E}\oplus\mathbb{C}^{n-s},\qquad\{e_1^2=e_1\}\oplus \mathbb{C}^{n-1}\qquad\text{or}\qquad\{e_1^2=e_1\}\oplus \widetilde{\E}\oplus\mathbb{C}^{n-s-1},\]
	where $\widetilde{\E}$ is an $s$-dimensional evolution algebra with maximal index of nilpotency and $\mathbb{C}^k$ denotes the $k$-dimensional zero evolution algebra.
\end{corollary}

\subsection{Idempotents}
As mentioned above, an element $u$ of an evolution algebra $\E$ is an idempotent when it satisfies $u^2=u$. Moreover, as stated in \cite{MQ_23_idempotents}, the existence of idempotent elements in an arbitrary evolution algebra is still an open problem. Given an evolution algebra $\E$ with natural basis $\{e_1,\dots,e_n\}$ and structure matrix $M_B(\E)=(a_{ij})$, an element $u=x_1e_1+\dots+x_ne_n$ is an idempotent if we have
\[u^2=\sum_{i=1}^{n}x_i^2e_i^2=\sum_{i=1}^{n}x_i^2\sum_{j=1}^na_{ij}e_j=\sum_{i=1}^nx_ie_i,\]
which yields the condition $\sum_{j=1}^{n}x_j^2a_{ji}=x_i$ for all $i=1,\dots,n$, or, equivalently,
\begin{equation}\label{sist_3}
	M_B(\mathcal{E})^t\begin{pmatrix}
		x_1^2 \\ \vdots \\ x_n^2
	\end{pmatrix}=\begin{pmatrix}
		x_1 \\ \vdots \\ x_n
	\end{pmatrix}.
\end{equation}
Note that \eqref{sist_3} is equivalent to \eqref{sist_2} when $M_B(\E)$ is invertible, i.e., when $\E$ is regular, and in this case idempotents correspond exactly to one-dimensional subalgebras. Thus, we have the following immediate consequence.
\begin{theorem}\label{th:idem}
	Every complex regular evolution algebra admits an idempotent.
\end{theorem}
We conclude by proposing a new conjecture concerning evolution algebras. Recall first that an evolution algebra $\E$ is called \textit{solvable} if there exists $n\in\mathbb{N}$ such that $\E^{(n)}=0$, where the sequence $\big(\E^{(k)}\big)$ is defined inductively by $\E^{(1)}=\E$ and $\E^{(k+1)}=\E^{(k)}\E^{(k)}$ for all $k\in\mathbb{N}$. Observe that the absence of idempotent elements is a necessary condition for an evolution algebra to be solvable. Indeed, if an idempotent $u\in\E$ exists, then $u\in\E^{(n)}$ for all $n\in\mathbb{N}$, which contradicts the assumption that $\E$ is solvable.
\begin{conjecture}\label{conj}
Let $\E$ be a complex evolution algebra. Then, the following assertions are equivalent:
\begin{enumerate}[\rm (i)]
\item $\E$ is solvable;
\item $\E$ admits no idempotents; and
\item the system~\eqref{sist_3} only admits the trivial solution.
\end{enumerate}
\end{conjecture}
Let us provide some evidence in support of Conjecture \ref{conj}.
\begin{itemize}
\item If $n=1$, then the conjecture is obviously true.
\item If $n=2$, we may rely on the classification given in \cite{CLOR_14}. If $\E$ is regular, Theorem~\ref{th:idem} already guarantees the existence of an idempotent. Hence, it remains to check that every non-solvable isomorphism class also contains an idempotent. 
According to the classification, the non-solvable algebras in dimension 2 are
\[\E_1\colon e_1^2=e_1, e_2^2=0\quad\text{and}\quad\E_2\colon e_1^2=e_2^2=e_1,\]
and both clearly contain the idempotent $e_1$. Thus, Conjecture~\ref{conj} holds for $n=2$.
\end{itemize}

Unlike the situation for nilpotent evolution algebras (see \cite[Theorem 3.4]{EL_15}), to our knowledge there is no characterisation of solvable evolution algebras. Therefore, if true, this conjecture would provide a structural characterisation of solvable evolution algebras.

\section*{Acknowledgments}
Supported by the projects PID2020-115155GB-I00 and PID2024-155502NB-I00 granted by MICIU/AEI/10.13039/501100011033; 
and by Xunta de Galicia through the Competitive Reference Groups (GRC), ED431C
2023/31.
The second author was also supported by the predoctoral contract FPU21/05685 from the Ministerio de Ciencia, Innovación y Universidades (Spain).

%
%
%
%


\end{document}